\documentstyle[11pt,amssymb,amsmath,amscd,amsbsy,amsfonts,amsthm,bbm]{article}

\input xy
\xyoption{all}

\newtheorem{thm}{Theorem}[section]
\newtheorem{lem}[thm]{Lemma}
\newtheorem{prop}[thm]{Proposition}
\newtheorem{cor}[thm]{Corollary}

\DeclareMathOperator{\Z}{\mathbb{Z}}
\DeclareMathOperator{\F}{\mathbb{F}}
\DeclareMathOperator{\one}{\mathbbm{1}}
\DeclareMathOperator{\SH}{\mathcal{SH}}
\DeclareMathOperator{\DM}{\mathcal{DM}}
\DeclareMathOperator{\X}{\mathfrak{X}}
\DeclareMathOperator{\A}{\mathcal{A}}
\DeclareMathOperator{\M}{\mathcal{M}}
\DeclareMathOperator{\G}{\mathcal{G}}
\DeclareMathOperator{\I}{\mathcal{I}}
\DeclareMathOperator{\T}{\mathcal{T}}
\DeclareMathOperator{\hl}{holim}
\DeclareMathOperator{\s}{\mathrm{S}}
\DeclareMathOperator{\hz}{\mathrm{H}{\mathbb Z}/2}
\DeclareMathOperator{\HM}{\mathrm{H}}
\DeclareMathOperator{\Mot}{\mathrm{M}}

\title{\textsc{Isotropic stable motivic homotopy groups of spheres}}
\author{Fabio Tanania}
\date{}

\begin{document}

\maketitle

\begin{abstract}
In this paper we explore the isotropic stable motivic homotopy category constructed from the usual stable motivic homotopy category, following the work of Vishik on isotropic motives (see \cite{V}), by killing anisotropic varieties. In particular, we focus on cohomology operations in the isotropic realm and study the structure of the isotropic Steenrod algebra. Then, we construct an isotropic version of the motivic Adams spectral sequence and apply it to find a complete description of the isotropic stable homotopy groups of the sphere spectrum, which happen to be isomorphic to the $Ext$-groups of the topological Steenrod algebra. At the end, we will see that this isomorphism is not only additive but respects higher products, completely identifying, as rings with higher structure, the cohomology of the classical Steenrod algebra with isotropic stable homotopy groups of spheres.	
	\end{abstract}

\section{Introduction}

A major problem in classical stable homotopy theory is the computation of the stable homotopy groups of the spheres. Many sophisticated tools have been implemented to tackle this problem during the years. One of this is for sure the Adams spectral sequence (see \cite{A}) which allows to obtain information on the stable homotopy groups of the sphere spectrum out of a purely algebraic object, namely the cohomology of the Steenrod algebra, which constitutes its $E_2$-page. As a consequence, this object has been deeply studied and much interest has arisen around its structure and properties with the final purpose of obtaining results on the stable homotopy groups.

In their groundbreaking work, Morel and Voevodsky have developed a homotopy theory for algebraic varieties, i.e. $A^1$-homotopy theory (see \cite{MV}), which has the enormous power of implementing topological techniques in the world of algebraic geometry. This unstable theory has, just as in topology, a stable counterpart, called stable motivic homotopy theory (see \cite{DLORV}). In this new and rich theory, it is possible to define a motivic sphere spectrum and to study its stable motivic homotopy groups, following the leading example provided by topology. As experience has shown, the motivic world is much more complicated than the topological one. Here cohomology theories are bigraded since there are two kinds of spheres, the simplicial one and the algebraic one, and even the motivic cohomology of the easiest object, namely the spectrum of the base field, is known to be much more complicated than its topological counterpart, i.e. the singular cohomology of a point. The motivic situation is as close as possible to the topological one if one considers algebraically closed fields, for example the complex numbers. In this case, topological realization functors come to help and many motivic phenomena can be studied with the help of topological insights. Nonetheless, even in this limit case, many exotic and non-classical phenomena appear, which definitely shows the deep richness of the motivic environment.

With the aim of studying stable motivic homotopy groups of spheres, techniques from topology have been adapted and new tools implemented. One of the key ingredients borrowed from topology for the computation of stable homotopy groups is the motivic Adams spectral sequence, a motivic version of its topological relative. This has been introduced by Morel in \cite{Mo2} and deeply studied by Dugger and Isaksen in \cite{DI}. Convergence of the spectral sequence has been analysed by Hu, Kriz and Ormsby in \cite{Hu} and \cite{Hu2}, by Kylling and Wilson in \cite{KW} and by Mantovani in \cite{Ma}. Then, stable homotopy groups of the motivic sphere spectrum have been studied over different base fields by several authors. Over complex numbers, there are important results by Dugger and Isaksen in \cite{DI}, by Guillou and Isaksen in \cite{GI} and by Andrews and Miller in \cite{AM}. Moreover, Levine showed in \cite{L} that over complex numbers the stable motivic homotopy groups of the sphere spectrum $\pi_{n,0}(\s)$ are isomorphic to the classical homotopy groups $\pi_n(\s)$. Over real numbers, there are results by Dugger and Isaksen in \cite{DI2} and by Guillou and Isaksen in \cite{GI2}. Over rational numbers there are results by Wilson in \cite{W} and over finite fields by Wilson and \O stv\ae r in \cite{WO}. Over general fields, there are available the complete description of the $0$th stem of the stable motivic homotopy groups of spheres by Morel in \cite{Mo} and of the $1$st stem by R\"ondigs, Spitzweck and \O stv\ae r in \cite{RSO}. On the other hand, the analysis of the motivic homotopy groups of the cofiber of $\tau$, the polynomial generator in bidegree $(1)[0]$ of the motivic cohomology of the point, over complex numbers by Isaksen in \cite{I} and by Gheorghe in \cite{G} have led to very interesting insights on both motivic homotopy theory and classical homotopy theory. In fact, these stable homotopy groups are identified as a ring with higher products with the $E_2$-page of the Adams-Novikov spectral sequence. This result is surprising since on the one hand it presents a fundamental object in classical homotopy theory, namely the $E_2$-page of the Adams-Novikov spectral sequence, as the motivic homotopy groups of a motivic $E_{\infty}$-ring spectrum and on the other shows that these homotopy groups are purely algebraic in nature. As we will see, we will achieve a similar result regarding the $E_2$-page of the classical Adams spectral sequence.

In order to study motivic categories and motivic problems over more general fields, one needs to find some substitutes of topological realization functors and find a way to simplify the original motivic categories without losing too much information. One possible solution is the one suggested by Vishik in his work on isotropic categories of motives (see \cite{V}). These isotropic versions of Voevodsky's triangulated category of motives (see \cite{V3}) are, at least for flexible fields, i.e. purely trascendental extensions of countable infinite degree, much simpler than the original one and permit to study and highlight some motivic phenomena by looking at them in an easier environment. More precisely, these isotropic categories are obtained by tensoring Voevodsky's categories with an idempotent that kills at once all motives of anisotropic varieties. Killing anisotropic varieties results into a simplification of the original category, but also in a deeper understanding of its structure. As an example, when computing the isotropic motivic cohomology of the point, Vishik obtains an exterior algebra in generators that are inverses of Milnor operations. This means that the Milnor subalgebra of the Steenrod algebra is encoded in the very fabric of isotropic categories.

In this paper, our purpose is to study an isotropic version of the stable motivic homotopy category, by following the path drawed by Vishik. After studying formal properties of this category and analysing cohomology operations in this isotropic environment we would like to move on to the study of isotropic stable homotopy groups of the sphere spectrum. In order to do so, we will need an isotropic version of the Adams spectral sequence, an analysis of its $E_2$-page and a study on its convergence. This will be achieved and the answer we get is exactly what one would expect, notably there exists an isotropic Adams spectral sequence with $E_2$-page provided by the cohomology of the isotropic Steenrod algebra and converging to something closely related to the isotropic motivic homotopy groups of the sphere spectrum, more precisely to the stable motivic homotopy groups of the $\hz$-nilpotent completion of the isotropic sphere spectrum $\X$, where $\hz$ is the Eilenberg-MacLane spectrum with $\Z/2$-coefficients and $\X$ is the unit of the isotropic stable motivic homotopy category (see Theorem \ref{lab}).

Once we get this important result we move on to the study of the structure of the isotropic $E_2$-page. Indeed, we get a more familiar description in terms of the classical Steenrod algebra. More precisely we prove that the $E_2$-page of the isotropic Adams spectral sequence is isomorphic on the slope $2$ line to the $E_2$-page of the topological Adams spectral sequence and trivial elsewhere (see Theorem \ref{e2}). This result implies not only the strong convergence of the spectral sequence due to the obvious vanishing regions but also forces all subsequent differentials to be trivial showing that the stable motivic homotopy groups of the isotropic sphere spectrum are exactly the $Ext$-groups of the topological Steenrod algebra (see Theorem \ref{main}).

Actually, more is true, since $\X$ happens to be an idempotent, and so a motivic $E_{\infty}$-ring spectrum. This allows us to identify at the end isotropic homotopy groups of spheres and the cohomology of the Steenrod algebra not only additively, as previously noted, but also as rings with all their higher products. This result tells us that the $Ext$ of the Steenrod algebra with its higher structure provided by Massey products is realised as motivic stable homotopy groups of a motivic $E_{\infty}$-ring spectrum with its higher structure provided by Toda brakets (see Corollary \ref{fin}). Moreover, we want to highlight that, as for the cofiber of $\tau$ over complex numbers, the motivic homotopy groups of $\X$ have a surprising purely algebraic nature.\\

\textbf{Outline.} We now briefly summarise the contents of each section of this paper. In Section $2$ we introduce isotropic categories, recall their construction and show some of their main features. Section $3$ is devoted to the analysis of cohomology operations in this isotropic environment. We define the isotropic Steenrod algebra and prove some results on its structure. Then, in Section $4$, we recall the construction of the motivic Adams spectral sequence and build an isotropic version of it. Moreover, we identify the isotropic $E_2$-page in terms of the cohomology of the isotropic Steenrod algebra and give results on convergence of the spectral sequence for the sphere spectrum. Section $5$ is devoted to the study of the isotropic $E_2$-page. In particular we simplify it till the point of getting an isomorphism with the cohomology of the topological Steenrod algebra on the slope $2$ line. This forces differentials to be trivial and allows the identification of stable motivic homotopy groups of the isotropic sphere spectrum with the $E_2$-page of the classical Adams spectral sequence. In Section $6$ we prove that this identification is not only additive but respects all higher products, i.e. sends Toda brackets to Massey products.\\

\textbf{Notation.} Here we fix some notations we will use throughout this paper.\\

\begin{tabular}{c|c}
	$k$ & flexible field, i.e. $k=k_0(t_1,t_2,\dots)$ for some field $k_0$, of characteristic different from 2\\
	$\SH(k)$ & stable motivic homotopy category over $k$\\
	$\SH(k/k)$ & isotropic stable motivic homotopy category over $k$\\
	$\DM(k)$ & triangulated category of motives with $\Z/2$-coefficients over $k$\\
	$\DM(k/k)$ & isotropic triangulated category of motives with $\Z/2$-coefficients over $k$\\
	$\pi_{**}(-)$ & stable motivic homotopy groups\\
	$\pi^{iso}_{**}(-)$ & isotropic stable motivic homotopy groups\\
	$H(-)$ & motivic cohomology with $\Z/2$-coefficients\\
	$H_{iso}(-)$ & isotropic motivic cohomology with $\Z/2$-coefficients\\
	$H(k)$ & motivic cohomology with $\Z/2$-coefficients of $Spec(k)$\\
	$H(k/k)$ & isotropic motivic cohomology with $\Z/2$-coefficients of $Spec(k)$\\
	$\A(k)$ & mod 2 motivic Steenrod algebra\\
	$\A(k/k)$ & mod 2 isotropic motivic Steenrod algebra\\
	$\A$ & mod 2 classical Steenrod algebra\\
	$\s$ & motivic sphere spectrum\\
	$\hz$ & motivic Eilenberg-MacLane spectrum with $\Z/2$-coefficients\\
	$\X$ & isotropic sphere spectrum\\
	$\one$ & unit of $\DM(k)$\\
\end{tabular}\\

Moreover, $\wedge$ denotes the smash product in $\SH(k)$ and $\otimes$ the tensor product in $\DM(k)$. We denote hom-sets in $\SH(k)$ by $[-,-]$ and the suspension $\s^{p,q} \wedge X$ of a motivic spectrum $X$ by $X^{p,q}$. For any motive $M$ in $\DM(k)$, we denote by $M(q)[p]$ the motive $M \otimes \one(q)[p]$. In general, we write bidegrees as $(q)[p]$, where $p$ is the topological degree and $q$ is the weight. \\

\textbf{Acknowledgements.} I would like to thank my PhD supervisor Alexander Vishik for proposing me the main topic of this paper and for many helpful conversations about it. Moreover, I would like to thank Jonas Irgens Kylling for very useful discussions and Paul Arne \O stv\ae r for having invited me in Oslo. I am also grateful to the referee for very helpful comments which allowed to improve the overall exposition and to simplify Sections 2 and 6.

\section{Isotropic motivic categories}

We would like to start, in this section, by introducing the main isotropic motivic categories, which constitute the environment where we will work in this paper. These tensor triangulated categories are obtained, roughly speaking, by tensoring the standard motivic categories with an idempotent that kills all anisotropic varieties. This operation on the one hand simplifies the original categories and on the other sheds light on some of their essential properties.

Two important triangulated categories in motivic homotopy theory are provided by the stable motivic homotopy category $\SH(k)$ (see \cite[Section 2.3]{DLORV}) and the triangulated category of motives with $\Z/2$ coefficients $\DM(k)$ (see \cite[Section 3.1]{V3}). The Dold-Kan correspondence extends to a pair of adjoint functors between $\SH(k)$ and $\DM(k)$
$$
\xymatrix{
	\SH(k) \ar@/_/[d]_{\Mot}\\
	\DM(k) \ar@/_/[u]_{\HM} 
}
$$
where $\Mot$ is the motivic functor and $\HM$ is the Eilenberg-MacLane functor.

The isotropic triangulated category of motives $\DM(k/k)$ has been constructed by Vishik in \cite{V}. A similar construction is possible for the isotropic stable motivic homotopy category $\SH(k/k)$. In a nutshell, the objects of $\SH(k/k)$ are the same of $\SH(k)$, while the morphisms are obtained by tensoring with a specific idempotent $\X$, i.e. 
$$Hom_{\SH(k/k)}(X,Y)=[\X \wedge X,\X \wedge Y]$$
where $\X$ is obtained as follows. First, let $Q$ be the disjoint union of all connected anisotropic varieties over $k$, namely varieties which do not have closed points of odd degree (it is a coproduct of infinitely many objects in the category of smooth schemes), and let $\check{C}(Q)$ be its \v{C}ech simplicial scheme, i.e. $\check{C}(Q)_n=Q^{n+1}$ and face and degeneracy maps are respectively given by partial projections and partial diagonals. The projection $\check{C}(Q) \rightarrow Spec(k)$ induces a distinguished triangle in $\SH(k)$:
$$\Sigma^{\infty}_+\check{C}(Q) \rightarrow \s \rightarrow \widetilde{\Sigma^{\infty}_+\check{C}(Q)} \rightarrow \Sigma^{\infty}_+\check{C}(Q)[1].$$
Then, the isotropic sphere spectrum $\X$ is defined as $\widetilde{\Sigma^{\infty}_+\check{C}(Q)}$. Moreover, let us denote the motivic spectrum $\Sigma^{\infty}_+\check{C}(Q)$ by $\X'$. The reason why one performs this construction resides in the fact that, for any anisotropic variety $P$, the map $\X' \wedge \Sigma^{\infty}_+P \rightarrow \Sigma^{\infty}_+P$ is an isomorphism (see \cite[Lemma 9.2]{V1}). This implies that $\X \wedge \Sigma^{\infty}_+P \cong 0$, which motivates the definition of the isotropic stable homotopy category. Note that $\X'$ is obviously an idempotent in $\SH(k)$, i.e. $\X' \wedge \X' \cong \X'$ (it follows from \cite[Section 3.1, Lemma 1.11]{MV}). Hence, we have that $\X' \wedge \X \cong 0$ and $\X$ is an idempotent as well.

We would like to point out that the reason why we work over a flexible base field is that, in this case, isotropic categories become particularly handy and informative. However, this choice is not restrictive at all, since as highlighted in \cite[Remark 1.2]{V} the restriction functors $\SH(k_0) \rightarrow \SH(k)$ and $\DM(k_0) \rightarrow \DM(k)$ are conservative, so no motivic information is lost in the passage. 

Let $L,C:\SH(k) \rightarrow \SH(k)$ be the exact functors defined by $L(X)=\X \wedge X$ and $C(X)=\X' \wedge X$ on objects and by $L(f)=id_{\X}\wedge f$ and $C(f)=id_{\X'} \wedge f$ on morphisms. First, we want to prove that $L$ is a localization functor and $C$ is a colocalization functor (see \cite[Definition 3.1.1]{HPS}).

\begin{prop}\label{loc}
$L$ is a localization functor and $C$ is a colocalization functor. Moreover, for any $X$ in $\SH(k)$ there is a distinguished triangle
$$C(X) \rightarrow X \rightarrow L(X) \rightarrow C(X)[1].$$
\end{prop}
\begin{proof}
First, note that for any $X$ 
$$C(X) \rightarrow X \rightarrow L(X) \rightarrow C(X)[1]$$
is a distinguished triangle by the very definition of $L$ and $C$.

By applying $L$ to the morphism $X \rightarrow \X \wedge X$ one obtains the isomorphism $\X \wedge X \rightarrow \X \wedge \X \wedge X$. This implies that the natural transformation $Li$, induced by the obvious natural transformation $i:1 \rightarrow L$, is an equivalence between $L$ and $L^2$.

Now, consider a morphism $f:\X' \wedge X \rightarrow \X \wedge Y$. By applying $L$ we get the square
$$
\xymatrix{
	\X' \wedge X \ar@{->}[r]^{f} \ar@{->}[d]&  \X \wedge Y \ar@{->}[d]^{\cong} \\
	\X \wedge \X' \wedge X \cong 0 \ar@{->}[r] & \X \wedge \X \wedge Y.
}
$$
Therefore, $f$ is the zero morphism and $[\X' \wedge X,\X \wedge Y]\cong 0$. By applying the cohomological functor $[-,\X \wedge Y]$ to the distinguished triangle
$$\X' \wedge X \rightarrow X \rightarrow \X \wedge X \rightarrow \X' \wedge X[1],$$
one obtains an isomorphism $[X,L(Y)] \cong [L(X),L(Y)]$.

Moreover, if $L(X) \cong 0$, then $L(X \wedge Y) \cong L(X) \wedge Y \cong 0$ for any $Y$. Hence, $L$ is a localization functor and $C$ is a colocalization functor by \cite[Lemma 3.1.6(a)]{HPS}, which completes the proof.
  
\end{proof}

At this point, we can consider two natural full subcategories of $\SH(k)$, namely $\T'=\X' \wedge \SH(k)$ and $\T=\X \wedge \SH(k)$. The second one is, by definition, nothing else but $\SH(k/k)$. First, we report the following result that summarizes the main properties of $\T'$ and $\T$.

\begin{prop}\label{tfae}
	The following are equivalent:\\
	1) $[X,Y] \cong 0$ for any $X$ in $\T'$;\\
	2) the morphism $Y \rightarrow \X \wedge Y$ induced by $\s \rightarrow \X$ is an isomorphism;\\
	3) $\X' \wedge Y \cong 0$;\\
	4) $Y \in \T$.
	
	Dually, the following are equivalent:\\
	1) $[X,Y] \cong 0$ for any $Y$ in $\T$;\\
	2) the morphism $\X' \wedge X \rightarrow X$ induced by $\X' \rightarrow \s$ is an isomorphism;\\
	3) $\X \wedge X \cong 0$;\\
	4) $X \in \T'$.
	
Moreover, $\T'$ is a localizing subcategory and $\T$ is a colocalizing subcategory. 
\end{prop}
\begin{proof}
See \cite[Lemma 3.1.6]{HPS}. 
\end{proof}

Note that, since $L$ is a smashing localization functor (see \cite[Definition 3.3.2]{HPS}), $\T$ is also a localizing subcategory. In the next result we prove that $\T$ is also closed under sequential homotopy limits.

\begin{prop}\label{hl}
	Consider an inverse system of objects in $\SH(k)$
	$$\dots \rightarrow X_i \rightarrow  \dots \rightarrow X_2 \rightarrow X_1 \rightarrow X_0.$$ 
	If $X_i$ belongs to $\T$ for all $i$, then $\hl X_i$ belongs to $\T$.
\end{prop}
\begin{proof}
	By definition of homotopy limit of an inverse system in triangulated categories we have a distinguished triangle
	$$\hl X_i \rightarrow \prod_i X_i \xrightarrow{id-sh} \prod_i X_i \rightarrow \hl X_i[1]$$
	in $\SH(k)$. Since $\T$ is closed under arbitrary products, we deduce that also $\hl X_i$ is in $\T$. 
	\end{proof}

Since essentially everything we have said comes from the fact that the idempotents $\X'$ and $\X$ induce a semi-orthogonal decomposition of $\SH(k)$, the same results and arguments hold for $\DM(k/k)$, whose objects are the same of $\DM(k)$ and morphisms are given by
$$Hom_{\DM(k/k)}(A,B)=Hom_{\DM(k)}(\Mot(\X) \otimes A,\Mot(\X) \otimes B).$$
More precisely, the respective motivic subcategories of $\DM(k)$ are $\Mot(\X') \otimes \DM(k)$ and $\Mot(\X) \otimes \DM(k)$. The second one is identified with the isotropic category of motives $\DM(k/k)$ and all results from Proposition \ref{loc} to Proposition \ref{hl} apply to $\DM(k/k)$ verbatim.

Note that there are obvious global-to-local functors $\SH(k) \rightarrow \SH(k/k)$ and $\DM(k) \rightarrow \DM(k/k)$ from motivic categories to their isotropic counterparts.

\begin{prop}
	The pair of adjoint functors $(\Mot,\HM)$ between $\SH(k)$ and $\DM(k)$ restricts to a pair of adjoint functors between $\SH(k/k)$ and $\DM(k/k)$
$$
\xymatrix{
	\SH(k/k) \ar@/_/[d]_{\Mot}\\
	\DM(k/k). \ar@/_/[u]_{\HM}
}
$$	
\end{prop}
\begin{proof}
The adjunction follows from the following chain of functorial equivalences:
\begin{align*}
Hom_{\DM(k/k)}(\Mot(X),B)& = Hom_{\DM(k)}(\Mot(\X) \otimes \Mot(X),\Mot(\X) \otimes B) &by \: definition\\
& \cong Hom_{\DM(k)}(\Mot(\X \wedge X),\Mot(\X) \otimes B) &by \:  K\ddot{u}nneth \: formula\\
& \cong Hom_{\SH(k)}(\X \wedge X,\HM(\Mot(\X) \otimes B)) &by \: adjunction\\
& \cong Hom_{\SH(k)}(\X \wedge X,\X \wedge \HM(B)) &by \: projection \: formula\\
& = Hom_{\SH(k/k)}(X,\HM(B)) &by \: definition
\end{align*} 
and the proof is complete. 
\end{proof}

The motivic cohomology and isotropic motivic cohomology of the point were computed respectively by Voevodsky in \cite{V1} over any field and by Vishik in \cite{V} over flexible fields. We report their results in the following two theorems.

\begin{thm}\label{voe}
	Let $k$ be a field of characteristic different from $2$. Then, the motivic cohomology of $Spec(k)$ with $\Z/2$-coefficients is described by
	$$H(k) \cong K^M(k)/2[\tau]$$
	where $\tau$ is the generator of $H^{0,1}(k) \cong \Z/2$ and $H^{n,n}(k) \cong K_n^M(k)/2$ is the Milnor K-theory of $k$ mod $2$.
\end{thm}
\begin{proof}
	See \cite[Theorem 6.1, Corollary 6.9 and Corollary 7.5]{V1}. 
\end{proof}

\begin{thm}\label{vis}
	Let $k$ be a flexible field. Then, for any $i \geq 0$ there exists a unique cohomology class $r_i$ of bidegree $(-2^i+1)[-2^{i+1}+1]$ such that
	$$H(k/k) \cong \Lambda_{\F_2}(r_i)_{i \geq 0}$$
	and $Q_jr_i=\delta_{ij}$, where $Q_j$ are the Milnor operations.
\end{thm}
\begin{proof}
	See \cite[Theorem 3.7]{V}. 
\end{proof}

It follows from Theorems \ref{voe} and \ref{vis}, by degree reasons, that the global-to-local map $H(k) \rightarrow H(k/k)$ is trivial everywhere but in bidegree $(0)[0]$ (see \cite[after Theorem 3.7]{V}). In particular, it sends $\tau$ to $0$. In this sense, the global-to-local functor looks like a realisation functor which is somehow perpendicular to topological realisation, which instead sends $\tau$ to $1$ (see \cite[after Definition 2.7]{DI}). Nonetheless, isotropic categories have encoded in their fabric some of the classical topological structures. As an example, Theorem \ref{vis} shows that the isotropic cohomology of the point is the dual of the Milnor subalgebra of the Steenrod algebra.

\section{The isotropic motivic Steenrod algebra}

In order to understand more deeply isotropic categories we would like to have a better grasp on cohomology operations in the isotropic context. In this section we study the isotropic Steenrod algebra and give some structural results which will be useful later on.

The motivic Steenrod algebra $\A(k)$ is the algebra of bistable motivic cohomology operations. It is defined by
$$\A^{p,q}(k)=[\hz,\hz^{p,q}].$$
Analogously, one can define the isotropic motivic Steenrod algebra $\A(k/k)$ by
$$\A^{p,q}(k/k)=Hom_{\SH(k/k)}(\hz,\hz^{p,q})=[\X \wedge \hz,\X \wedge \hz^{p,q}].$$

Killing the Milnor K-theory and mapping $\tau$ to $1$ or to $0$ give two possible quotients of the motivic Steenrod algebra. As we have already noticed in the previous section, the first quotient reflects topological realisation while the second is closer to isotropic realisation. Let us study both these situations more closely.

On the one hand, we obtain the topological Steenrod algebra $\A$ endowed with the morphism $\A(k) \rightarrow \A$ which sends $\tau$ to $1$ and $x$ to $0$ for any $x \in K^M_{>0}(k)/2$. On the other hand, we can define the generalised Steenrod algebra $\A_0$ as the quotient of $\A(k)$ modulo its two-sided ideal generated by the elements of $H(k)$ of bidegree different from $(0)[0]$, and we get a morphism $\A(k) \rightarrow \A_0$ which sends $\tau$ to $0$ and $x$ to $0$ for any $x \in K^M_{>0}(k)/2$.

The dual of the motivic Steenrod algebra has the following structure (see \cite[Theorem 5.6]{HKO} and \cite[Theorem 12.6]{V2}):
$$\A(k)_* \cong {\frac {H(k)[\tau_0,\tau_i,\xi_i]_{i \geq 1}} {(\tau_i^2+\tau \xi_{i+1}+\rho(\tau_{i+1}+\tau_0 \xi_{i+1}))}}$$
where $\tau_i$ has bidegree $(2^i-1)[2^{i+1}-1]$, $\xi_i$ has bidegree $(2^i-1)[2^{i+1}-2]$ for any $i$ and $\rho$ is the class of $-1$ in $H_{-1,-1}(k) \cong K^M_1(k)/2$. Note that in the previous isomorphism $H(k)$ is intended with homological degrees, i.e. $H(k)_{p,q}=H(k)^{-p,-q}$.

It follows that the dual of the topological Steenrod algebra is described by 
$$\A_* \cong \F_2[\tau_i]_{i \geq 0}$$
while the dual of the generalised Steenrod algebra is described by
$$(\A_0)_* \cong \Lambda_{\F_2}(\tau_i)_{i \geq 0} \otimes_{\F_2} \F_2[\xi_j]_{j \geq 1}.$$
The coproduct in $(\A_0)_*$ is given by (see \cite[Lemma 12.11]{V2}):
$$\psi(\xi_k)=\sum_{i=0}^k \xi_{k-i}^{2^i} \otimes \xi_i;$$
$$\psi(\tau_k)=\sum_{i=0}^k \xi_{k-i}^{2^i} \otimes \tau_i+\tau_k\otimes 1.$$

Note that the dual of the generalised Steenrod algebra has the same structure as the dual of the topological Steenrod algebra at odd prime $p$ (see \cite[Theorem 2 and Theorem 3]{M}) but with coefficients in $\F_2$. Indeed, the results in \cite[Section $6$]{M} which are stated for odd primes can be generalised to the prime $2$, providing the same results for the generalised Steenrod algebra just by substituting $\F_p$-coefficients with $\F_2$-coefficients. We recall now the main results we will need in this paper.

The dual of the generalised Steenrod algebra $(\A_0)_*$ is an ${\mathbb F}_2$-vector space with basis given by $\{\tau^E\xi^R\}_{E,R}$ where $E=(e_0,\dots,e_n)$ with $e_i \in \{0,1\}$ and $R=(r_1,\dots,r_m)$ with $r_j \in \Z_{\geq 0}$ and $\tau^E\xi^R=\Pi_{i=0}^n\tau_i^{e_i}\Pi_{j=1}^m\xi_j^{r_j}$. We will denote by $\widehat{\tau^E\xi^R}$ the dual in $\A_0$ of $\tau^E\xi^R$ with respect to the monomial basis of $(\A_0)_*$. Following \cite{M}, one can define some special operations in $\A_0$, namely the Milnor operations $Q_i=\widehat{\tau_i}$ and the operations $P^R=\widehat{\xi^R}$. Moreover, denote by $Q^E$ the product $\Pi_{i=0}^n Q_i^{e_i}$, where $E=(e_0,\dots,e_n)$ with $e_i \in \{0,1\}$.

\begin{thm}\label{mil}
	The elements $Q^EP^R$ form an additive basis for $\A_0$
	which is dual to the basis $\{\tau^E\xi^R\}_{E,R}$ for $(\A_0)_*$. The elements $Q_i$ generate an exterior subalgebra of $\A_0$ which we denote by $\M$ and call the Milnor subalgebra:
	$$\M=\Lambda_{\F_2}(Q_i)_{i \geq 0}.$$
	Milnor operations commute with the elements $P^R$ according to the rule:
	\begin{equation}\label{eq}
	[Q_k,P^R]=Q_{k+1}P^{R-(2^k,0,\dots)}+Q_{k+2}P^{R-(0,2^k,\dots)}+\dots \: . 
	\end{equation}
\end{thm}
\begin{proof}
The proof is the same of \cite[Theorem 4a]{M} with ${\mathbb F}_2$-coefficients instead of ${\mathbb F}_p$-coefficients. 
\end{proof}

The basis introduced by the previous theorem is called the Milnor basis for $\A_0$.

\begin{cor}
The elements $Q_i$ can be defined inductively in $\A_0$ by the recursive formulas:
$$Q_0=Sq^1;$$
$$Q_{i+1}=Sq^{2^{i+1}}Q_i+Q_iSq^{2^{i+1}}.$$
\end{cor}
\begin{proof}
	See \cite[Corollary 2]{M}. 
\end{proof}

Theorem \ref{mil} implies that $\A_0$ is a free left module over $\M$ with basis given by $\{P^R\}_{R}$. In the next proposition we use formula \ref{eq} to construct another basis for the generalised Steenrod algebra.

\begin{prop}\label{basis}
$\{P^RQ^E\}_{R,E}$ is a basis for $\A_0$ as an ${\mathbb F}_2$-vector space.
\end{prop}
\begin{proof}
First, we want to prove that $\{P^RQ^E\}_{R,E}$ generates $\A_0$. In order to do so, it is enough to show that any $Q^EP^R$ belongs to the ${\mathbb F}_2$-vector space generated by $\{P^RQ^E\}_{R,E}$.

Let us start by showing that any $Q_iP^R$ belongs to the ${\mathbb F}_2$-vector space generated by $\{P^RQ^E\}_{R,E}$. Define the following order in $\oplus_{j=1}^{\infty}\Z_{\geq 0}$: $S \leq R$ if and only if $s_1 \leq r_1$, $s_2 \leq r_2$, $\dots$ where $r_j=0$ for almost all $j$. We want to prove our first statement by a simultaneous induction on $R$, for all $i$. If $R=\overline{0}$, there is nothing to prove. Now, suppose that the statement holds for any $S<R$. Then,
$$Q_iP^R=P^RQ_i+Q_{i+1}P^{R-(2^i,0,\dots)}+Q_{i+2}P^{R-(0,2^i,\dots)}+\dots=P^RQ_i+Y$$
where $Y$ is expressible in terms of $\{P^RQ^E\}_{R,E}$ by induction hypothesis. Hence, $Q_iP^R$ is expressible too.

At this point, we want to prove the main statement. We proceed by induction on the number of $1$s in $E$. If $E=\overline{0}$ there is nothing to prove. Now suppose the statement is true for $F$ such that $E=F+(0,\dots,0,1,0,\dots)$ where the $1$ in the second summand is in the $i$-th position where $F$ has a $0$. Now, we have that
$$Q^EP^R=Q_iQ^FP^R=Q_i\sum P^SQ^G$$
by induction hypothesis. But we have already proven that each $Q_iP^S$ is expressible in terms of $\{P^RQ^E\}_{R,E}$, therefore $Q^EP^R$ is expressible too.

In order to conclude that $\{P^RQ^E\}_{R,E}$ is indeed a basis, we just notice that in each degree there is the same number of $P^RQ^E$ and $Q^EP^R$. 
\end{proof}

From the previous proposition it immediately follows that $\A_0$ is also a free right module over $\M$ with basis given by $\{P^R\}_{R}$.

Now, let $X$ be an infinite matrix
$$ 
\begin{pmatrix} 
* & x_{01} & x_{02} & \dots \\
x_{10} & x_{11} & \dots & \\
x_{20} & \dots & & \\
\vdots & & &
\end{pmatrix}
$$
with nonnegative entries almost all $0$, $R(X)=(r_1,r_2,\dots)$ where $r_i=\sum_j2^jx_{ij}$, $S(X)=(s_1,s_2,\dots)$ where $s_j=\sum_ix_{ij}$, $T(X)=(t_1,t_2,\dots)$ where $t_n=\sum_{i+j=n}x_{ij}$ and 
$$b(X)={\frac {\Pi t_n!}{\Pi x_{ij}!}}.$$

\begin{thm}
	In $\A_0$ the following product formula holds:
$$P^RP^S=\sum_{X|R(X)=R,S(X)=S}b(X)P^{T(X)}.$$
\end{thm}
\begin{proof}
The proof is the same of \cite[Theorem 4b]{M} with ${\mathbb F}_2$-coefficients instead of ${\mathbb F}_p$-coefficients. 
\end{proof}

It follows from the previous theorem that the ${\mathbb F}_2$-vector space generated by $\{P^R\}_R$ is a subalgebra of $\A_0$ which we denote by $\G$. 

\begin{prop}\label{2si}
There are isomorphisms
$$\A_0 \cong \M \otimes_{\F_2} \G \cong \G \otimes_{\F_2} \M.$$
In particular, the right ideal generated by the elements of $\M$ of positive degree is two-sided in $\A_0$ and the quotient of $\A_0$ modulo it is isomorphic to $\G$.
\end{prop}
\begin{proof}
	It follows immediately from Proposition \ref{basis} and from the commutator formula \ref{eq}. 
\end{proof}

Note that, by degree reasons, $Sq^{2r}$ belongs to $\G$ for any $r \geq 0$, since Milnor operations lie above the slope 2 line, i.e. $\M^{p,q}$ is trivial for $2q \geq p >0$ . The next result tells us that, indeed, $\G$ is nothing else but the classical Steenrod algebra.

\begin{prop}\label{iso}
	There is an isomorphism of algebras $\A \rightarrow \G$ mapping $Sq^r$ to $Sq^{2r}$ for any $r \geq 0$.
\end{prop}
\begin{proof}
	First, note that, since $\A(k)$ is generated as an $H(k)$-algebra by the Steenrod squares $Sq^r$, then also $\A_0$ is generated as an $\F_2$-algebra by the Steenrod squares $Sq^r$. It follows that $\G$ is generated as an algebra by $\{Sq^{2r}\}_{r \geq 0}$ since all operations in $\G$ lie on the slope $2$ line, i.e. $\G^{p,q}$ is trivial when $p \neq 2q$. Moreover, the only relations among Steenrod squares in $\G$ are the Adem relations in $\A(k)$ involving even squares, namely
	$$Sq^{2r}Sq^{2s}=\sum_{c=0}^r {\binom {2s-c-1}{2r-2c}}\tau^{c \: mod2}Sq^{2r+2s-c}Sq^c$$
	in which we impose $\tau=0$. It follows that the following are the only relations holding in $\G$
	$$Sq^{2r}Sq^{2s}=\sum_{t=0}^{[r/2]} {\binom {2s-2t-1}{2r-4t}}Sq^{2r+2s-2t}Sq^{2t}.$$
	 Looking at binary presentations it is easy to see that 
	 $${\binom {2s-2t-1}{2r-4t}}={\binom {s-t-1}{r-2t}} \: mod \: 2$$
	 from which it follows that in $\G$
	 $$Sq^{2r}Sq^{2s}=\sum_{t=0}^{[r/2]} {\binom {s-t-1}{r-2t}}Sq^{2r+2s-2t}Sq^{2t}.$$
	 Now, recall that $\A$ is the algebra generated by Steenrod squares $\{Sq^r\}_{r \geq 0}$ subject to the classical Adem relations
	 $$Sq^rSq^s=\sum_{t=0}^{[r/2]} {\binom {s-t-1}{r-2t}}Sq^{r+s-t}Sq^t.$$
	 It immediately follows that the homomorphism $\A \rightarrow \G$ mapping $Sq^r$ to $Sq^{2r}$ is well defined and, indeed, an isomorphism. 
\end{proof}

Once we understood a bit more what the projection of $\tau$ to $0$ implies on the level of cohomology operations, let us give a description of the isotropic Steenrod algebra.

\begin{prop}\label{st}
	Let $k$ be a flexible field. Then, there exists an isomorphism of $H(k/k)-\A_0$-bimodules
	$$\A(k/k) \cong H(k/k) \otimes_{\F_2} \A_0.$$	
\end{prop}
\begin{proof}
	By definition and adjunction
	$$\A^{p,q}(k)=[\hz,\hz^{p,q}] \cong Hom_{\DM(k)}(\Mot(\hz),\one(q)[p])$$
	and 
	$$\A^{p,q}(k/k)=Hom_{\SH(k/k)}(\hz,\hz^{p,q}) \cong Hom_{\DM(k/k)}(\Mot(\hz),\one(q)[p]).$$
	Recall from \cite[Corollary 3.33]{Vm} and \cite[Theorem 3.2]{HKO} that $\Mot(\hz)$ is a split proper Tate motive. Hence,
	$$\A(k/k) \cong H(k/k) \otimes_{H(k)} \A(k).$$
	Since the morphism $H(k) \rightarrow H(k/k)$ is trivial everywhere but in bidegree $(0)[0]$, we have that
	$$\A(k/k) \cong H(k/k) \otimes_{\F_2} \A_0$$
	which is what we aimed to show.	
\end{proof}

The following proposition gives the whole action of $\A(k/k)$ on $H(k/k)$.

\begin{prop}
	Let $k$ be a flexible field. Then, the action of the Steenrod squares over the isotropic cohomology of the point is described by
	$$Sq^{2^j}r_i=
	\begin{cases}
	r_{i-1}, & i=j \\
	0, & i \neq j
	\end{cases}
	.$$
\end{prop}
\begin{proof}
	Since $\rho=0$, from \cite[Remark 5]{K} we know that
	$$[Q_i,Sq^{2^j}]=
	\begin{cases}
	0, & i \geq j \\
	Q_{i+1}Sq^{2^j-2^{i+1}}, & i < j
	\end{cases}
	.$$
	Note that $Sq^{2^j}r_i$ is in bidegree $(2^{j-1}-2^i+1)[2^j-2^{i+1}+1]$.
	
	If $i < j$, then $Sq^{2^j}r_i=0$ since it belongs to the first quadrant where the isotropic cohomology of the point is trivial.
	
	If $i=j$, then we have that
	$$1=Q_ir_i=Sq^{2^i}Q_{i-1}r_i+Q_{i-1}Sq^{2^i}r_i=Q_{i-1}Sq^{2^i}r_i$$
	from which we deduce that $Sq^{2^i}r_i$ is non-trivial. Hence, $Sq^{2^i}r_i=r_{i-1}$ since it is the only non-zero class in bidegree $(-2^{i-1}+1)[-2^i+1]$.
	
	If $i > j$, then $Sq^{2^j}r_i=0$ since the only classes in bidegrees $(q)[p]$ such that $p=2q-1$ are the $r_k$ and none of them is in the same bidegree of $Sq^{2^j}r_i$. This completes the proof. 
\end{proof}

\section{The isotropic motivic Adams spectral sequence}

Our main aim in this paper is to compute isotropic stable motivic homotopy groups of spheres and relate them to some classical object in stable homotopy theory. We want to achieve this by using an isotropic version of the motivic Adams spectral sequence.
 
First, let us recall how the motivic Adams spectral sequence is constructed (see \cite{A} and \cite{DI}). We start with the standard Adams resolution
	$$
\xymatrix{
	\dots \ar@{->}[r] & (\overline{\hz})^{\wedge i+1} \ar@{->}[r] \ar@{->}[d]	 &(\overline{\hz})^{\wedge i} \ar@{->}[r] \ar@{->}[d] &  \dots \ar@{->}[r]   & \overline{\hz} \ar@{->}[r] \ar@{->}[d]	 & \s \ar@{->}[d]  \\
	& \hz \wedge (\overline{\hz})^{\wedge i+1} \ar@{->}[ul]^{[1]} &	\hz \wedge (\overline{\hz})^{\wedge i} \ar@{->}[ul]^{[1]} & & \hz \wedge \overline{\hz} \ar@{->}[ul]^{[1]}  &	\hz \ar@{->}[ul]^{[1]} 
}
$$ 
where $\s \rightarrow \hz$ is the standard unit map from the sphere spectrum to the Eilenberg-MacLane spectrum, $\overline{\hz}$ is defined as $Cone(\s \rightarrow \hz)[-1]$ and each triangle
$$(\overline{\hz})^{\wedge i+1} \rightarrow (\overline{\hz})^{\wedge i} \rightarrow \hz \wedge (\overline{\hz})^{\wedge i} \rightarrow (\overline{\hz})^{\wedge i+1}[1]$$
is a cofiber sequence in $\SH(k)$. Then, after smashing everything with a motivic spectrum $X$ one gets the following Postnikov system:
	$$
\xymatrix{
	\dots \ar@{->}[r] &  (\overline{\hz})^{\wedge i} \wedge X \ar@{->}[r] \ar@{->}[d] &  \dots \ar@{->}[r]   & \overline{\hz} \wedge X \ar@{->}[r] \ar@{->}[d]	 & X \ar@{->}[d]  \\
 &	\hz \wedge (\overline{\hz})^{\wedge i} \wedge X \ar@{->}[ul]^{[1]} & & \hz \wedge \overline{\hz} \wedge X \ar@{->}[ul]^{[1]}  &	\hz \wedge X. \ar@{->}[ul]^{[1]} 
}
$$
If we apply motivic homotopy groups $\pi_{*,*'}(-)=[\s^{*,*'},-]$ to the previous Postnikov system, we get an unrolled exact couple 
	$$
\xymatrix{
	\dots \ar@{->}[r] &\pi_{*,*'}((\overline{\hz})^{\wedge i}\wedge X) \ar@{->}[r] \ar@{->}[d] &  \dots \ar@{->}[r]   & \pi_{*,*'}(\overline{\hz}\wedge X) \ar@{->}[r] \ar@{->}[d]	 & \pi_{*,*'}(X) \ar@{->}[d]  \\
	&	\pi_{*,*'}(\hz \wedge (\overline{\hz})^{\wedge i}\wedge X) \ar@{->}[ul]^{[1]} & & \pi_{*,*'}(\hz \wedge \overline{\hz}\wedge X) \ar@{->}[ul]^{[1]}  &	\pi_{*,*'}(\hz\wedge X) \ar@{->}[ul]^{[1]} 
}
$$ 
that induces a spectral sequence with $E_1$-page given by
$$E_1^{s,t,u}=\pi_{t-s,u}(\hz \wedge (\overline{\hz})^{\wedge s}\wedge X)$$
and first differential
$$d_1^{s,t,u}:E_1^{s,t,u}=\pi_{t-s,u}(\hz \wedge (\overline{\hz})^{\wedge s}\wedge X) \rightarrow E_1^{s+1,t,u}=\pi_{t-s-1,u}(\hz \wedge (\overline{\hz})^{\wedge s+1}\wedge X).$$
This spectral sequence is called motivic Adams spectral sequence and has differentials on the $E_r$-page given by
$$d_r:E_r^{s,t,u} \rightarrow E_r^{s+r,t+r-1,u}.$$ 
It converges to the stable homotopy groups of a motivic spectrum strongly related to $X$, i.e. its $\hz$-nilpotent completion. Let us recall now how nilpotent completions of motivic spectra are obtained. First, define $\overline{\hz}_n$ as the cone of the morphism $(\overline{\hz})^{\wedge n+1} \rightarrow \s$, i.e. there are distinguished triangles
$$(\overline{\hz})^{\wedge n+1} \rightarrow \s \rightarrow \overline{\hz}_n \rightarrow (\overline{\hz})^{\wedge n+1}[1].$$
Then, by octahedron axiom we get maps $\overline{\hz}_{n+1} \rightarrow \overline{\hz}_n$. This way one gets an inverse system of objects in $\SH(k)$
$$\dots \rightarrow \overline{\hz}_n \rightarrow \dots \rightarrow \overline{\hz}_2 \rightarrow \overline{\hz}_1 \rightarrow \overline{\hz}_0$$
and, after smashing it with a motivic spectrum $X$, one obtains the following inverse system:
$$\dots \rightarrow \overline{\hz}_n \wedge X \rightarrow \dots \rightarrow \overline{\hz}_2 \wedge X \rightarrow \overline{\hz}_1\wedge X \rightarrow \overline{\hz}_0\wedge X.$$
Then, the $\hz$-nilpotent completion of $X$ is defined as the homotopy limit of the previous inverse system, i.e. $X^{\wedge}_{\hz}=\hl \overline{\hz}_n \wedge X$ (see \cite[Proposition 5.5]{Bo}). Moreover, note that, by results of Mantovani (see \cite[Theorems 1.0.1 and 1.0.3]{Ma}), over perfect fields, for connective motivic spectra, $\hz$-nilpotent completions, $\hz$-localizations and $(2,\eta)$-completions coincide.

All in all, the convergence of the motivic Adams spectral sequence is described in the following result.

\begin{thm}
If $\varprojlim_{r}^1 E_r^{s,t,u}=0$ for any $s,t,u$, then the motivic Adams spectral sequence is strongly convergent to the stable motivic homotopy groups of the $\hz$-nilpotent completion of the motivic spectrum $X$.	
\end{thm}
\begin{proof}
See \cite[Proposition 6.3]{Bo} and \cite[Remark 6.11]{DI}. 
\end{proof}

The condition on the triviality of $\varprojlim_{r}^1 E_r^{s,t,u}$ is satisfied if $\{E_r^{s,t,u}\}$ is Mittag-Leffler, i.e. for each $s,t,u$ there exists $r_0$ such that $E_r^{s,t,u} \cong E_{\infty}^{s,t,u}$ whenever $r > r_0$ (see \cite[after Proposition 6.3]{Bo}).

Now, let us focus our attention on the motivic Adams spectral sequence for the sphere spectrum. In this case, the $E_1$-page can be described in the following way:
$$E_1^{s,t,u}=\pi_{t-s,u}(\hz \wedge (\overline{\hz})^{\wedge s}) \cong Hom_{\A(k)}^{t-s,u}(H(\hz \wedge (\overline{\hz})^{\wedge s}),H(k))$$
$$ \cong Hom_{\A(k)}^{t,u}(\A(k) \otimes_{H(k)} \overline{\A(k)}^{\otimes s},H(k)).$$
It is worth noticing that the previous chain of equivalences is due to the fact that $\hz \wedge \hz$ is equivalent to a motivically finite type wedge of $\hz$ (see \cite[Definition 2.11, Definition 2.12 and Lemma 6.3]{DI}), from which it follows by an easy induction argument that all spectra $\hz \wedge (\overline{\hz})^{\wedge s}$ are motivically finite type wedges of $\hz$. Then, the cohomology rings $H(\hz \wedge (\overline{\hz})^{\wedge s})$ are identified with the $\A(k)$-modules $A(k) \otimes_{H(k)} \overline{\A(k)}^{\otimes s}$. In practice, the $E_1$-page of the motivic Adams spectral sequence can be obtained as follows. First, one chooses a free $\A(k)$-resolution of $H(k)$
$$0 \leftarrow H(k) \leftarrow \A(k) \leftarrow \A(k) \otimes_{H(k)} \overline{\A(k)} \leftarrow \dots \leftarrow \A(k) \otimes_{H(k)} \overline{\A(k)}^{\otimes s} \leftarrow \dots \: .$$
Then, applying the functor $Hom_{\A(k)}(-,H(k))$, one obtains the complex
$$Hom_{\A(k)}(\A(k),H(k)) \rightarrow Hom_{\A(k)}(\A(k) \otimes_{H(k)} \overline{\A(k)},H(k)) \rightarrow \dots$$ 
$$\rightarrow Hom_{\A(k)}(\A(k) \otimes_{H(k)} \overline{\A(k)}^{\otimes s},H(k)) \rightarrow \dots$$
whose terms constitute the $E_1$-page and whose morphisms are the differentials $d_1$ of the spectral sequence. The homology of this complex is essentially the $E_2$-page of the spectral sequence, i.e. $Ext_{\A(k)}(H(k),H(k))$. Hence, one has the following result.

\begin{thm}
	Let $k$ be an algebraically closed field. Then, the motivic Adams spectral sequence for the sphere spectrum is strongly convergent to the stable motivic homotopy groups of the $\hz$-nilpotent completion of the sphere spectrum $\s^{\wedge}_{\hz}$ and has $E_2$-page described by
	$$E_2^{s,t,u}=Ext^{s,t,u}_{\A(k)}(H(k),H(k)) \Longrightarrow \pi_{t-s,u}(\s^{\wedge}_{\hz})$$
	where $s$ is the homological degree and $(t,u)$ is the internal bidegree ($t$ is the topological degree and $u$ is the weight).
\end{thm}
\begin{proof}
	See \cite[Corollary 6.15]{DI}. 
\end{proof}

Strong convergence of the motivic Adams spectral sequence for the sphere spectrum is known also over fields with finite virtual cohomological dimension (see \cite{KW}).

We can project the standard Adams resolution to the isotropic category $\SH(k/k)$. This way, we get an isotropic motivic Adams spectral sequence with $E_1$-page given by
$$E_1^{s,t,u}=\pi_{t-s,u}^{iso}(\hz \wedge (\overline{\hz})^{\wedge s} \wedge X)$$
and first differential
$$d_1^{s,t,u}:\pi_{t-s,u}^{iso}(\hz \wedge (\overline{\hz})^{\wedge s} \wedge X) \rightarrow \pi_{t-s-1,u}^{iso}(\hz \wedge (\overline{\hz})^{\wedge s+1} \wedge X)$$
where $\pi^{iso}_{*,*'}$ are the isotropic stable motivic homotopy groups functors, namely
$$\pi^{iso}_{*,*'}(X)=Hom_{\SH(k/k)}(\s^{*,*'},X)=[\X \wedge \s^{*,*'},\X \wedge X] \cong [\s^{*,*'},\X \wedge X]=\pi_{*,*'}(\X \wedge X)$$
for any spectrum $X$, where the isomorphism in the previous chain of equivalences comes from Proposition \ref{loc}. Also in this case, differentials on the $E_r$-page have tri-degrees given by
$$d_r:E_r^{s,t,u} \rightarrow E_r^{s+r,t+r-1,u}.$$
In fact, note that the isotropic motivic Adams spectral sequence for the motivic spectrum $X$ is nothing else but the usual motivic Adams spectral sequence for the motivic spectrum $\X \wedge X$. If we focus on the sphere spectrum, then we have an $E_1$-page given by
$$E_1^{s,t,u}=\pi^{iso}_{t-s,u}(\hz \wedge (\overline{\hz})^{\wedge s}).$$
We would like to get a description for this $E_1$-page similar to the one of the motivic Adams spectral sequence, in order to get a handy $E_2$-page expressible algebraically in terms of $Ext$-groups over the isotropic motivic Steenrod algebra. Indeed, we have the following situation. Let $V$ be the graded ${\mathbb F}_2$-vector space defined as $\bigoplus_{\beta} \F_2^{-p_{\beta},-q_{\beta}}$ where $\{(p_{\beta,}q_{\beta})\}_{\beta}$ is a motivically finite type set of bidegrees, i.e. for any $\alpha$ there are finitely many $\beta$ such that $p_{\alpha} \geq p_{\beta}$ and $2q_{\alpha} -p_{\alpha} \geq 2q_{\beta} -p_{\beta}$. Moreover, denote by $\hz(V)$ the generalised Eilenberg-MacLane spectrum $\bigvee_{\beta} \hz^{p_{\beta},q_{\beta}}$. Then, one has the following commutative diagram:
$$
\xymatrix{
	\pi_{*,*'}^{iso}(\hz(V)) \ar@{->}[r] \ar@{->}[d]&  Hom^{*,*'}_{\A(k/k)}(H_{iso}(\hz(V)),H(k/k)) \ar@{->}[d]  \\
	Hom^{*,*'}_{\F_2}(V,H(k/k)) \ar@{->}[r] & Hom^{*,*'}_{\A(k/k)}(\A(k/k) \otimes_{\F_2} V,H(k/k))
}
$$
where by $H_{iso}(X)$ we mean the isotropic motivic cohomology of $X$, i.e. $H_{iso}^{p,q}(X)=[\X \wedge X, \X \wedge \hz^{p,q}]$.

The left vertical arrow is the isomorphism described by the following chain of equivalences:
$$\pi_{*,*'}^{iso}(\hz(V)) \cong [\s^{*,*'},\X \wedge \bigvee_{\beta} \hz^{p_{\beta},q_{\beta}}] \cong [\s^{*,*'},\bigvee_{\beta} \X \wedge \hz^{p_{\beta},q_{\beta}}] \cong \bigoplus_{\beta}[\s^{*,*'},\X \wedge \hz^{p_{\beta},q_{\beta}}]\cong$$ $$\bigoplus_{\beta}H(k/k)^{p_{\beta}-*,q_{\beta}-*'} \cong \prod_{\beta}H(k/k)^{p_{\beta}-*,q_{\beta}-*'} \cong \prod_{\beta} Hom^{*,*'}_{\F_2}(\F_2^{-p_{\beta},-q_{\beta}},H(k/k)) \cong$$
$$Hom^{*,*'}_{\F_2}(\bigoplus_{\beta} \F_2^{-p_{\beta},-q_{\beta}},H(k/k))=Hom^{*,*'}_{\F_2}(V,H(k/k))$$
where the isomorphism $\bigoplus_{\beta}H(k/k)^{p_{\beta}-*,q_{\beta}-*'} \cong \prod_{\beta}H(k/k)^{p_{\beta}-*,q_{\beta}-*'}$ follows from the fact that for any bidegree $(*')[*]$ only for a finite number of $\beta$ the group $ H(k/k)^{p_{\beta}-*,q_{\beta}-*'}$ is non-trivial.
The bottom horizontal map is obviously an isomorphism since $\A(k/k)$ is an ${\mathbb F}_2$-vector space. The right vertical map is an isomorphism since
$$Hom^{*,*'}_{\A(k/k)}(H_{iso}(\hz(V)),H(k/k))=Hom^{*,*'}_{\A(k/k)}(H_{iso}(\bigvee_{\beta} \hz^{p_{\beta},q_{\beta}}),H(k/k)) \cong$$
$$Hom^{*,*'}_{\A(k/k)}(\prod_{\beta}H_{iso}(\hz^{p_{\beta},q_{\beta}}),H(k/k)) = Hom^{*,*'}_{\A(k/k)}(\prod_{\beta} \A(k/k)^{-p_{\beta},-q_{\beta}},H(k/k)) \cong$$
$$ Hom^{*,*'}_{\A(k/k)}(\bigoplus_{\beta} \A(k/k)^{-p_{\beta},-q_{\beta}},H(k/k)) \cong Hom^{*,*'}_{\A(k/k)}(\A(k/k) \otimes_{\F_2} V,H(k/k)).$$
In the previous chain of equivalences the identification between $Hom^{*,*'}_{\A(k/k)}(\prod_{\beta} \A(k/k)^{-p_{\beta},-q_{\beta}},H(k/k))$ and $ Hom^{*,*'}_{\A(k/k)}(\bigoplus_{\beta} \A(k/k)^{-p_{\beta},-q_{\beta}},H(k/k))$ comes, as before, from the fact that for any bidegree $(*')[*]$ only for a finite number of $\beta$ the group $Hom^{*,*'}_{\A(k/k)}( \A(k/k)^{-p_{\beta},-q_{\beta}},H(k/k))\cong H(k/k)^{p_{\beta}-*,q_{\beta}-*'}$ is non-trivial.

It follows from the previous considerations and by recalling that $\hz \wedge (\overline{\hz})^{\wedge s}$ are motivically finite type wedges of $\hz$ that the $E_1$-page of the isotropic Adams spectral sequence can be identified in the following way: 
$$E_1^{s,t,u} \cong Hom_{\A(k/k)}^{t-s,u}(H_{iso}(\hz \wedge (\overline{\hz})^{\wedge s}),H(k/k)) \cong Hom_{\A(k/k)}^{t,u}(\A(k/k) \otimes_{H(k/k)} \overline{\A(k/k)}^{\otimes s},H(k/k)).$$
In other words, the $E_1$-page of the isotropic motivic Adams spectral sequence can be obtained as follows. First, one chooses a free $\A(k/k)$-resolution of $H(k/k)$
$$0 \leftarrow H(k/k) \leftarrow \A(k/k) \leftarrow \A(k/k) \otimes_{H(k/k)} \overline{\A(k/k)} \leftarrow \dots \leftarrow \A(k/k) \otimes_{H(k/k)} \overline{\A(k/k)}^{\otimes s} \leftarrow \dots \: .$$
Then, applying the functor $Hom_{\A(k/k)}(-,H(k/k))$, one obtains the complex
$$Hom_{\A(k/k)}(\A(k/k),H(k/k)) \rightarrow Hom_{\A(k/k)}(\A(k/k) \otimes_{H(k/k)} \overline{\A(k/k)},H(k/k)) \rightarrow \dots$$ 
$$\rightarrow Hom_{\A(k/k)}(\A(k/k) \otimes_{H(k/k)} \overline{\A(k/k)}^{\otimes s},H(k/k)) \rightarrow \dots$$
whose terms constitute the $E_1$-page and whose morphisms are the differentials $d_1$ of the spectral sequence. The homology of this complex is essentially the $E_2$-page of the isotropic Adams spectral sequence, i.e. $Ext_{\A(k/k)}(H(k/k),H(k/k))$. Therefore, we have the following result on the convergence of the isotropic Adams spectral sequence for the sphere spectrum.

\begin{thm}\label{lab}
	If $\varprojlim_{r}^1 E_r^{s,t,u}=0$ for any $s,t,u$, then the isotropic motivic Adams spectral sequence for the sphere spectrum is strongly convergent to the stable motivic homotopy groups of the $\hz$-nilpotent completion of the isotropic sphere spectrum $\X^{\wedge}_{\hz}$ and has $E_2$-page described by
	$$E_2^{s,t,u}=Ext^{s,t,u}_{\A(k/k)}(H(k/k),H(k/k)) \Longrightarrow \pi_{t-s,u}(\X^{\wedge}_{\hz}).$$
\end{thm}
\begin{proof}
Since the isotropic motivic Adams spectral sequence for the sphere spectrum is nothing else but the motivic Adams spectral sequence for the spectrum $\X$, it converges to $[\X \wedge \s^{*,*'},\X^{\wedge}_{\hz}]$. By definition, $\X^{\wedge}_{\hz}=\hl \X \wedge \overline{\hz}_n$. Hence, by Proposition \ref{hl} we have that $\X^{\wedge}_{\hz}$ belongs to $\SH(k/k)$. It immediately follows that the isotropic Adams spectral sequence converges to 
$$[\X \wedge \s^{*,*'},\X^{\wedge}_{\hz}] \cong [\s^{*,*'},\X^{\wedge}_{\hz}]=\pi_{*,*'}(\X^{\wedge}_{\hz})$$
which is what we aimed to show. 
\end{proof}

\section{Isotropic homotopy groups of the sphere spectrum}

The last theorem in the previous section suggests us that, in order to get some structural results on the isotropic homotopy groups of spheres, we need to describe the $E_2$-page of the isotropic Adams spectral sequence in a more familiar way. This is exactly the purpose of this section and we will achieve it by studying the algebraic properties of the isotropic Steenrod algebra, its Milnor subalgebra and some modules over it.

First, recall that, by Proposition \ref{2si}, $\G$ is the quotient ring of $\A_0$ modulo its two-sided ideal generated by all Milnor operations. Hence, we have a tautological functor from the category of left $\G$-modules to the category of left $\A_0$-modules which is obviously exact and fully faithful. Then, we can compose this functor with the functor sending a left $\A_0$-module $N$ to the left $\A(k/k)$-module $\A(k/k) \otimes_{\A_0} N$. This functor is exact since $\A(k/k)$ is free over $\A_0$. All in all, we obtain an exact functor from the category of left $\G$-modules to the category of left $\A(k/k)$-modules
$$N \in \G-{\mathrm Mod} \mapsto \A(k/k) \otimes_{\A_0} N \in \A(k/k)-{\mathrm Mod}$$
which induces homomorphisms between $Ext$-modules
$$Ext^i_{\G}(N,N') \rightarrow Ext^i_{\A(k/k)}(\A(k/k) \otimes_{\A_0} N,\A(k/k) \otimes_{\A_0} N')$$
commuting with Yoneda product (and higher products). We want to show that, under mild restrictions on $N'$, these homomorphisms are isomorphisms. In order to proceed we need the following well known Baer's criterion for graded rings.

\begin{prop}
	(Baer's criterion) Let $R$ be a graded ring and let $M$ be a graded left $R$-module. Then, $M$ is injective if and only if, for any homogeneous left ideal $I$ of $R$ and any graded morphism $\phi:I \rightarrow M$ of left $R$-modules, there exists a graded morhism $\psi:R \rightarrow M$ extending $\phi$.
\end{prop}
\begin{proof}
	See \cite[Corollary 2.4.8]{N}. 
\end{proof}

Before moving on to the main results, we first need some terminology. For any left $\G$-module $N$ and any integer $i$ let $N \langle i \rangle $ be the left $\G$-submodule of $N$ whose elements are all the elements of $N$ lying on the slope $2$ line of the equation $p=2q+i$, where $p$ and $q$ represent respectively the usual topological degree and weight. Note that $N$ is isomorphic to $\bigoplus_{i \in \Z} N \langle i \rangle $ as left $\G$-modules. We say that $N$ is bounded above if there exists a $p'$ such that $N^{p,q} \cong 0$ for all $p \geq p'$. Clearly, if $N$ is bounded above all $N \langle i \rangle$ are bounded above as well.

We are now ready to prove that the left $\A(k/k)$-modules that are images of bounded above left $\G$-modules via the above mentioned functor are injective as left $\M$-modules.

\begin{prop}\label{inj}
	Let $k$ be a flexible field. Then, for any bounded above left $\G$-module $N$, $\A(k/k) \otimes_{\A_0}N$ is an injective left $\M$-module.
\end{prop}
\begin{proof}
	In order to prove the injectivity of $\A(k/k) \otimes_{\A_0}N$, we want to use Baer's criterion. So, let $\I$ be a left ideal of $\M$ and $\phi:\I \rightarrow \A(k/k) \otimes_{\A_0}N$ a graded morphism of left $\M$-modules. We want to find a $\psi:\M \rightarrow \A(k/k) \otimes_{\A_0}N$ extending $\phi$.
	
	First, consider a bounded above left $\G$-module $N$ that is concentrated on a single slope $2$ line, i.e. $N \cong N \langle i \rangle$ for some $i$. Note that, by Proposition \ref{st}, $\A(k/k) \otimes_{\A_0}N \cong H(k/k) \otimes_{\F_2} N$ and $\M$ acts on the left only on $H(k/k)$. Let $x=Q_{I_x}=\Pi_{i \in I_x}Q_i$ be the product of some Milnor operations and, similarly, let $r_{I_x}$ denote the product $\Pi_{i \in I_x}r_i$. Consider the set $A=\{x=Q_{I_x} \in \I : 0 \neq \phi(x) \in N \}$. Note that $A$ is finite, possibly empty, since $N$ is bounded above by hypothesis. At this point, define a morphism $\psi:\M \rightarrow \A(k/k) \otimes_{\A_0}N$ by imposing $\psi(1)=\sum_{x \in A} r_{I_x} \phi(x)$. We just need to check that this morphism extends $\phi$.
	
	Take an element $y=Q_{I_y}$ in $\I$ and let $\phi(y)=\sum_ {\beta \in B} r_{I_{\beta}}n_{\beta}$ for some non-trivial $r_{I_{\beta}} \in H(k/k)$ and $n_{\beta} \in N$, where the set $B$ may be empty. Note that for any $\beta \in B$ the elements $r_{I_{\beta}}$ are all on the same slope $2$ line by degree reasons, from which it follows that $I_{\gamma} \not \subseteq I_{\beta}$ for any couple of different elements $\beta$ and $\gamma$ in $B$. Hence, we have that for any $\gamma \in B$
	$$\phi(Q_{I_{\gamma}}y)= Q_{I_{\gamma}}\phi(y)=Q_{I_{\gamma}}\sum_ {\beta \in B} r_{I_{\beta}}n_{\beta}=n_{\gamma}.$$
	Therefore, we deduce that $I_{\gamma} \cap I_y = \emptyset$ for any $\gamma \in B$ and
	$$\phi(y)=\sum_ {\beta \in B} r_{I_{\beta}}\phi(Q_{I_{\beta}}y)=\sum_ {\beta \in B} r_{I_{\beta}} \phi(Q_{I_{\beta} \cup I_y}).$$
	On the other hand,
	$$\psi(y)=\psi(Q_{I_y})=Q_{I_y}\psi(1)=Q_{I_y}\sum_{x \in A} r_{I_x} \phi(x)=\sum_{x \in A,I_y \subseteq I_x} r_{I_x-I_y} \phi(Q_{I_x}).$$
	Clearly, each summand of $\phi(y)$ is a summand of $\psi(y)$ too, since $Q_{I_{\beta} \cup I_y} \in A$ for any $\beta \in B$. Now, take $x=Q_{I_x} \in A$ such that $I_y \subseteq I_x$. Then, we have that
	$$\phi(x)=Q_{I_x-I_y}\phi(y)=Q_{I_x-I_y}\sum_ {\beta \in B} r_{I_{\beta}}n_{\beta} \in N$$
	from which it follows that there exists $\gamma \in B$ such that $\phi(x)=n_{\gamma}$ and $I_x-I_y=I_{\gamma}$. Therefore, each summand of $\psi(y)$ is a summand of $\phi(y)$ too from which we deduce that $\psi(y)=\phi(y)$.
	
    Now, let $N$ be any bounded above left $\G$-module. Then, as noticed before, $N \cong \bigoplus_{i \in \Z} N \langle i \rangle$ as left $\G$-modules. Let $\phi_i$ be the composition of $\phi$ with the projection to the $i$th factor $\A(k/k) \otimes_{\A_0}N \langle i \rangle$. From what we have already proved we know that there exist morphisms $\psi_i: \M \rightarrow \A(k/k) \otimes_{\A_0}N \langle i \rangle$ extending the respective $\phi_i$. Therefore, there exists a morphism $\psi': \M \rightarrow \prod_{i \in \Z} A(k/k) \otimes_{\A_0}N\langle i \rangle$ compatible with the $\psi_i$. Now, note that, since $N$ is bounded above, $\psi'(1)$ belongs to $\bigoplus_{i \in \Z} A(k/k) \otimes_{\A_0}N\langle i \rangle$. Hence, $\psi'$ lifts to a morphism $\psi: \M \rightarrow \bigoplus_{i \in \Z}A(k/k) \otimes_{\A_0} N\langle i \rangle \cong A(k/k) \otimes_{\A_0}N$ extending $\phi$, which completes the proof. 
\end{proof}

At this point, let us show a lemma which will be useful in the proof of the main theorem of this section.

\begin{lem}\label{tl}
	Let $k$ be a flexible field. Then, for any left $\G$-modules $N$ and $N'$ the obvious homomorphisms
	$$Hom_{\F_2}(N,N') \rightarrow Hom_{\M}(N,N') \rightarrow Hom_{\M}(N,\A(k/k) \otimes_{\A_0} N')$$ 
	and
	$$Hom_{\G}(N,N') \rightarrow Hom_{\A_0}(N,N') \rightarrow Hom_{\A_0}(N,\A(k/k) \otimes_{\A_0} N')$$ 
	are isomorphisms.
\end{lem}
\begin{proof}
First, note that, since $\G$ is the quotient of $\A_0$ by its two-sided ideal generated by elements of $\M$ of positive degree by Proposition \ref{2si}, then $Hom_{\F_2}(N,N') \rightarrow Hom_{\M}(N,N')$ and $Hom_{\G}(N,N') \rightarrow Hom_{\A_0}(N,N')$ are clearly isomorphisms.

Let $f:N \rightarrow \A(k/k) \otimes_{\A_0} N'$ be a morphism of left $\M$-modules, let $x$ be an element of $N$ and $f(x)=\sum_{\beta \in B} r_{I_{\beta}}n_{\beta}$. Suppose there exists $\beta \in B$ such that $I_{\beta} \neq \emptyset$. Choose among them $\gamma$ with the maximal $I_{\gamma}$. Then, there exists an $i_{\beta} \in I_{\gamma}-I_{\beta}$, for any $\gamma \neq \beta \in B$. Hence, $Q_{I_{\gamma}}r_{I_{\beta}}=Q_{I_{\gamma}-\{i_{\beta}\}}Q_{i_{\beta}}r_{I_{\beta}}=0$ for any $\beta \neq \gamma$ by Theorem \ref{vis}. It follows that
$$0=f(Q_{I_{\gamma}}x)=Q_{I_{\gamma}}f(x)=Q_{I_{\gamma}}\sum_{\beta \in B} r_{I_{\beta}}n_{\beta}=\sum_{\beta \in B} Q_{I_{\gamma}}(r_{I_{\beta}})n_{\beta}=n_{\gamma}.$$
Thus, $f$ can only take values in $N'$.

Now, let $f:N \rightarrow \A(k/k) \otimes_{\A_0} N'$ be a morphism of left $\A_0$-modules. Then, in particular, it is a morphism of left $\M$-modules. From what we have already proved it follows that $f$ takes values in $N'$, which is what we aimed to show. 
\end{proof}

The next theorem is the main result of this section. In fact, it allows to present the isotropic $E_2$-page as a more familiar object, as we will point out in the corollary.

\begin{thm}
	Let $k$ be a flexible field. Then, the homomorphism
	$$Ext^q_{\G}(N,N') \rightarrow Ext^q_{\A(k/k)}(\A(k/k) \otimes_{\A_0} N,\A(k/k) \otimes_{\A_0} N')$$
	is an isomorphism for any $q$, any left $\G$-module $N$ and any bounded above left $\G$-module $N'$.
\end{thm}
\begin{proof}
	First, we have that 
	$$Ext^q_{\A(k/k)}(\A(k/k) \otimes_{\A_0} N,\A(k/k) \otimes_{\A_0} N') \rightarrow Ext^q_{\A_0}(N,\A(k/k) \otimes_{\A_0} N')$$
	is an isomorphism since $\A(k/k)$ is free over $\A_0$. Then, we only need to prove that the obvious composition
	$$Ext^q_{\G}(N,N') \rightarrow Ext^q_{\A_0}(N,N') \rightarrow Ext^q_{\A_0}(N,\A(k/k) \otimes_{\A_0} N')$$
	is an isomorphism.

    Let us define by induction certain left $\A_0$-modules:
	$$N_0=N;$$
	$$N_p=Ker(\A_0 \otimes_{\M} N_{p-1} \rightarrow N_{p-1}) \cong Ker(\G \otimes_{\F_2} N_{p-1} \rightarrow N_{p-1}).$$
	Notice that, by Proposition \ref{2si}, all $N_p$ have trivial left $\M$-action. By definition, for any $p>0$ we have a short exact sequence
	$$0 \rightarrow N_p \rightarrow \A_0 \otimes_{\M} N_{p-1} \rightarrow N_{p-1} \rightarrow 0$$
	that induces the following morphism of long exact sequences:
	$$
	\xymatrix@C=0.6em{
		 \ar@{->}[r] &Ext^q_{\G}(N_{p-1},N') \ar@{->}[r] \ar@{->}[d]& Ext^q_{\G}(\G \otimes_{\F_2} N_{p-1},N') \ar@{->}[r] \ar@{->}[d]&  Ext^q_{\G}(N_p,N') \ar@{->}[d] \ar@{->}[r] & \\
		 \ar@{->}[r] &Ext^q_{\A_0}(N_{p-1},\A(k/k) \otimes_{\A_0} N') \ar@{->}[r] & Ext^q_{\A_0}(\A_0 \otimes_{\M} N_{p-1},\A(k/k) \otimes_{\A_0} N') \ar@{->}[r]& Ext^q_{\A_0}(N_p,\A(k/k) \otimes_{\A_0} N') \ar@{->}[r] &.
	}
	$$
	Since $\G$ is free over $\F_2$, we have that
	$$Ext^q_{\G}(\G \otimes_{\F_2} N_{p-1}, N') \cong Ext^q_{\F_2}(N_{p-1},N') \cong
	\begin{cases}
	Hom_{\F_2}(N_{p-1},N'), & q=0 \\
	0, & q > 0
	\end{cases}
	.$$
	Analogously, since $\A_0$ is free over $\M$, $\A(k/k) \otimes_{\A_0} N'$ is an injective left $\M$-module by Proposition \ref{inj} and from Lemma \ref{tl}, one obtains that
	$$Ext^q_{\A_0}(\A_0 \otimes_{\M} N_{p-1},\A(k/k) \otimes_{\A_0} N') \cong Ext^q_{\M}(N_{p-1},\A(k/k) \otimes_{\A_0} N') \cong
	\begin{cases}
	Hom_{\F_2}(N_{p-1},N'), & q=0 \\
	0, & q > 0
	\end{cases}
	.$$
	It follows that, if $q>0$, then we have the following commutative diagram
	$$
	\xymatrix{
		Ext^1_{\G}(N_{q-1},N') \ar@{->}[r]^{\cong} \ar@{->}[d]& Ext^{q}_{\G}(N,N') \ar@{->}[d]  \\
		Ext^1_{\A_0}(N_{q-1},\A(k/k) \otimes_{\A_0} N') \ar@{->}[r]^{\cong} & Ext^{q}_{\A_0}(N,\A(k/k) \otimes_{\A_0} N')
	}
$$
where the horizontal maps are isomorphisms. Moreover, we have a morphism of exact sequences
	$$
	\xymatrix@C=1.7em{
		Ext^0_{\F_2}(N_{q-1},N') \ar@{->}[r] \ar@{->}[d]_{\cong}& Ext^0_{\G}(N_q,N') \ar@{->}[r] \ar@{->}[d]^{\cong}&  Ext^1_{\G}(N_{q-1},N') \ar@{->}[r] \ar@{->}[d]&0 \ar@{->}[d]^{\cong}  \\
		Ext^0_{\M}(N_{q-1},\A(k/k) \otimes_{\A_0} N') \ar@{->}[r] & Ext^0_{\A_0}(N_q,\A(k/k) \otimes_{\A_0} N') \ar@{->}[r]& Ext^1_{\A_0}(N_{q-1},\A(k/k) \otimes_{\A_0} N') \ar@{->}[r] &0
	}
	$$
	where the vertical isomorphisms are due to Lemma \ref{tl}, from which it follows that the third vertical morphism is an isomorphism as well. Therefore, the morphism 
	$$Ext^q_{\G}(N,N') \rightarrow Ext^q_{\A_0}(N,\A(k/k) \otimes_{\A_0} N')$$
	is an isomorphism, which finishes the proof.
	 
\end{proof}

As an immediate corollary of the previous theorem we get the following result.

\begin{thm}\label{e2}
	Let $k$ be a flexible field. Then, the homomorphism 
	$$Ext_{\G}(\F_2,\F_2) \rightarrow Ext_{\A(k/k)}(H(k/k),H(k/k))$$
	is an isomorphism of rings respecting higher products.
\end{thm}
\begin{proof}
	It follows immediately by applying the previous theorem to the case $N \cong N' \cong \F_2$ and by noticing that $\A(k/k) \otimes_{\A_0} \F_2 \cong H(k/k)$ by Proposition \ref{st}. The compatibility with products comes from the fact that the isomorphism is induced by the exact functor described at the beginning of this section. 
\end{proof}

The previous theorem can be rephrased in the following way.

\begin{cor}\label{boh}
	Let $k$ be a flexible field. Then, the $E_2$-page of the isotropic motivic Adams spectral sequence coincides with the $E_2$-page of the classical Adams spectral sequence in doubled degree.
\end{cor}
\begin{proof}
	Recall from Theorem \ref{e2} that the $E_2$-page of the isotropic motivic Adams spectral sequence is given by
	$$E_2=Ext_{\A(k/k)}(H(k/k),H(k/k)) \cong Ext_{\G}(\F_2,\F_2)$$
	while the $E_2$-page of the classical Adams spectral sequence is given by
	$$E_2=Ext_{\A}(\F_2,\F_2).$$
	The statement follows immediately from the fact that $\A$ and $\G$ are isomorphic via the map $Sq^r \mapsto Sq^{2r}$ by Proposition \ref{iso}. 
\end{proof}

Once the $E_2$-page of the isotropic Adams spectral sequence has been identified the game is done. Indeed, the very structure of it implies that all subsequent differentials in the spectral sequence are trivial, as stated in the following result.

\begin{thm}\label{main}
	Let $k$ be a flexible field. Then, the stable motivic homotopy groups of the $\hz$-completed isotropic sphere spectrum are completely described by
	$$\pi_{*,*'}(\X^{\wedge}_{\hz}) \cong Ext_{\G}^{2*'-*,2*',*'}(\F_2,\F_2) \cong Ext_{\A}^{2*'-*,*'}(\F_2,\F_2).$$
\end{thm}
\begin{proof}
Note that all the differentials $d_r:E_r^{s,t,u} \rightarrow E_r^{s+r,t+r-1,u}$ from the second on vanish since $E_2^{s,t,u}=0$, and so $E_r^{s,t,u}=0$ for $r \geq 2$, when $t \neq 2u$. Therefore, we have that $\{E_r^{s,t,u}\}$ is Mittag-Leffler, since $E_2^{s,t,u} \cong E_{\infty}^{s,t,u}$ for any $s,t,u$. Hence, the isotropic motivic Adams spectral sequence is strongly convergent and collapses at the second page, from which the statement follows immediately. 
\end{proof}

As a corollary we have the following vanishing result.

\begin{prop}\label{van}
	Let $k$ be a flexible field. Then, $\pi_{p,q}(\X^{\wedge}_{\hz})  \cong 0$ for any $p>2q$, or $p<q$, or $q<0$.
\end{prop}
\begin{proof}
	The vanishing regions are due to Theorem \ref{main}, since 
	$$\pi_{p,q}(\X^{\wedge}_{\hz})  \cong Ext_{\A}^{2q-p,q}(\F_2,\F_2)$$ 
	and the cohomology of the Steenrod algebra $Ext_{\A}^{s,t}(\F_2,\F_2)$ vanishes when $s<0$ or $t-s<0$, i.e. $p>2q$ or $p<q$. Since $Ext_{\A}^{s,t}(\F_2,\F_2)\cong 0$ also for $t<0$, we have that the isotropic homotopy groups vanish also for $q<0$, which completes the proof. 
\end{proof}

\section{Higher structure of isotropic homotopy groups}

The isomorphism of Theorem \ref{main} establishes  a connection between the isotropic motivic world and the classical topological one. More precisely, a very important object in classical homotopy theory, namely the cohomology of the classical Steenrod algebra, is realised as the stable motivic homotopy groups of the $\hz$-completed isotropic sphere spectrum. Note that this isomorphism is at the moment only additive, but it is well known that the cohomology of the Steenrod algebra, being the cohomology of a dga, carries a higher structure in the form of Massey products. What we would like to know is if there exists a higher structure on $\X$ inducing a higher structure on its stable motivic homotopy groups which coincides with that of the cohomology of the Steenrod algebra. This way the isomorphism obtained in the previous section would be not only additive, but it would respect also higher products.

\begin{prop}\label{rob}
$\X$ is a motivic $E_{\infty}$-ring spectrum.
\end{prop}
\begin{proof}
	It follows from \cite[Proposition 4.8.2.9]{Lu} since the map $\s \rightarrow \X$ exhibits $\X$ as an idempotent object of $\SH(k)$. 
\end{proof}

Since $\X$ is an $E_{\infty}$-ring spectrum, the motivic Adams spectral sequence converging to its stable motivic homotopy groups is multiplicative (see \cite[Theorem 2.3.3]{Ra}). By recalling that it collapses at the second page and by the Moss convergence theorem (see \cite[Theorem 1.2]{Mos} and \cite[Theorem 3.1.1]{I}), we deduce that the isomorphism
$$Ext_{\A(k/k)}^{s,t,u}(H(k/k),H(k/k)) \cong \pi_{t-s,u}(\X^{\wedge}_{\hz})$$
respects the higher structure. By combining this remark with Theorem \ref{e2} and Corollary \ref{boh}, one deduces the following result.

\begin{cor}\label{fin}
	Let $k$ be a flexible field. Then, the isomorphism
	$$\pi_{*,*'}(\X^{\wedge}_{\hz}) \cong Ext_{\A}^{2*'-*,*'}(\F_2,\F_2)$$
	is an isomorphism of rings sending Toda brackets in $\pi_{*,*'}$ to Massey products in $Ext$.
\end{cor}

\footnotesize{
	}
	
\noindent {\scshape Mathematisches Institut, Ludwig-Maximilians-Universit\"at M\"unchen}\\
fabio.tanania@gmail.com

\end{document}